\documentclass[11pt,reqno]{amsart}
\usepackage[letterpaper, portrait, margin=.9in, headheight=0pt]{geometry}

\usepackage{mathtools}
\usepackage{enumerate}
\usepackage{amsmath,amsthm,amssymb,amsfonts}
\usepackage{stmaryrd}
\usepackage{hyperref}
\usepackage[normalem]{ulem}
\usepackage{soul}
\usepackage{xcolor}
\usepackage{tikz}
\usepackage{tikz-cd}
\usetikzlibrary{calc, positioning, fit, shapes.misc,arrows,bending}
\usepackage{subcaption}
\usepackage{blkarray}
\usepackage{cleveref}
\newtheorem{thm}{Theorem} 
\newtheorem{cor}[thm]{Corollary}
\theoremstyle{definition}

\newtheorem{defi}[thm]{Definition}

\newtheorem{que}[thm]{Question}

\makeatletter
\DeclareRobustCommand{\cev}[1]{%
  \mathpalette\do@cev{#1}%
}
\newcommand{\do@cev}[2]{%
  \fix@cev{#1}{+}%
  \reflectbox{$\m@th#1\vec{\reflectbox{$\fix@cev{#1}{-}\m@th#1#2\fix@cev{#1}{+}$}}$}%
  \fix@cev{#1}{-}%
}
\newcommand{\fix@cev}[2]{%
  \ifx#1\displaystyle
    \mkern#23mu
  \else
    \ifx#1\textstyle
      \mkern#23mu
    \else
      \ifx#1\scriptstyle
        \mkern#22mu
      \else
        \mkern#22mu
      \fi
    \fi
  \fi
}

\makeatother

\newcommand{\dvecEG}{%
  \mathrel{\vbox{\offinterlineskip\ialign{%
    ##\hfil\cr
    $\scriptscriptstyle\hspace{.2ex}\leftrightarrows$\cr
    \noalign{\kern-.1ex}
    $E(G)$\cr
    }}}}

\usepackage{titlesec}

\titleformat{\section}
       {\normalfont
       \fontsize{14}{17}\bfseries}{\thesection}{1em}{}

\titleformat{\subsection}
       {\normalfont\fontsize{12}{17}\bfseries}{\thesubsection}{1em}{}

\title{A note on kernel-perfect orientations and DP-colorings from derangement assignments}

\author{Ian Gossett}

\begin{document}

\setstcolor{red}

\begin{abstract} 
We prove a generalization of the well-known Bondy-Bopanna-Siegel Lemma to DP-colorings of graphs from a class of correspondence assignments that we call \textit{derangement assignments}. Since DP-colorings from derangement assignments generalize zero-free list colorings of signed graphs, this yields an orientation theorem for zero-free signed list colorings, as well. 
\end{abstract}
\maketitle
\section{Introduction}
 DP-colorings of graphs, originally called correspondence colorings, are a generalization of list colorings that were introduced by Dvorak and Postle in \cite{MR3758240}. (DP-colorings also generalize several other previously studied types of graph labelings. See, e.g., \cite{MR4233786} for a discussion.) One of the main research goals pertaining to DP-colorings has been to determine similarities and differences between list colorings and DP-colorings (e.g. \cite{MR4262023},\cite{MR3518419},\cite{MR3889157},\cite{MR3758240}). In particular, two well-known orientation theorems for list colorings, the Alon-Tarsi theorem \cite{MR1179249} and the Bondy-Boppana-Siegel lemma (see \cite{MR1179249},\cite{MR1309363}), are known not to apply to DP-colorings `as-is,' in the sense that they do not hold true if we replace list colorings with DP-colorings in their statements.

 Even so, in \cite{gossett2023orientation}, several analogs of the Alon-Tarsi theorem were proved for certain types of DP-colorings. 
 In this note, we prove a generalization of the Bondy-Boppana-Siegel lemma to DP-colorings from a particular class of correspondence assignments, which we call \textit{derangement assignments}. We then give an application of this theorem to signed list colorings of signed graphs, as defined in \cite{JIN2016234}.
 
\section{Background and Definitions}
\subsection{Basic Definitions and the Bondy-Boppana-Siegel Lemma}

 By a \textit{graph}, we will mean a finite, loopless multigraph and by \textit{digraph} we will mean a finite, loopless multidigraph, though we will sometimes be explicit about this for the sake of clarity. By \textit{loopless}, we mean that no vertex is self adjacent. If $G=(V,E)$ is a graph, and $e\in E$, we write $e= \{v,w\}$ to specify that $v$ and $w$ are the endpoints of $e$.  (This is a slight abuse of notation, since the endpoints do not determine $e$ uniquely.) Similarly, when $D=(V,\vec{E})$ is a digraph, we write $\vec{e}=(v,w)$ to express that $\vec{e}$ is a directed edge from $v$ to $w$.

If $G$ is a graph, we write $d_G(v)$ to denote the degree of a vertex $v$ of $G$. If $D$ is a digraph, we let $d^+_D(v)$ denote the out-degree of $v$ in $D$, and let $d^-_D(v)$ denote the in-degree of $v$ in $D$. In this note, we consider \textit{biorientations} of graphs, where edges of a graph are allowed to be oriented in both directions; a biorientation of a graph $G=(V,E)$ is a directed graph $D=(V,\vec{E})$ such that $V(D)=V(G)$, and for each $e=\{u,v\}\in E(G)$ there is an edge $\vec{e}=(u,v)\in \vec{E}(D)$, $\cev{e}=(v,u)\in \vec{E}(D)$, or both,  and if $u$ and $v$ are non-adjacent in $G$, they are non-adjacent in $D$. If $G$ is a graph and $A\subseteq V$, denote the induced subgraph of $G$ whose vertex set is $A$ by $G[A]$ (and similarly by $D[A]$ for digraphs). Given a biorientation $D=(V,\vec{E})$ of $G$, define $E_{D2}(G)\subseteq E(G)$ to be set of all edges $e$ of $G$ such that $e$ is oriented in both directions in $D$.

\begin{defi}
    Let $G$ be graph and $L=\{L(v)\}_{v\in V(G)}$ be an assignment of lists to the vertices of $G$. If there exists a function $\varphi:V\rightarrow \bigcup_{v\in V}{L(v)} $ such that $\varphi(v)\in L(v)$ for each $v\in V$, and $\varphi(u)\neq \varphi(v)$ for all $\{u,v\}\in E(G)$, then $\varphi$ is called an $ {L}$\textit{-coloring} of $G$, and we say that $G$ is $ {L}$\textit{-colorable}. 
\end{defi}

\begin{defi}
    Let $G=(V,E)$ be a graph, and suppose that $f:V(G)\rightarrow \mathbb{N}$. $G$ is $ {f}$\textit{-choosable} if for any assignment of lists $L=\{L(v)\}_{v\in V}$ such that $|L(v)|\geq f(v)$ for each $v$, there exists an $L$-coloring of $G$. We say that $G$ is $ {k}$\textit{-choosable} if $G$ is $f$-choosable for the constant function $f(v)=k$. 
\end{defi}
\begin{defi}
Let $D=(V,\vec{E})$ be a digraph. A subset $U\subseteq V$ is a \textit{kernel} of $D$ if $U$ is an independent set and for each $v\in V\setminus U$, there exists some $u\in U$ such that $(v,u)\in \vec{E}$. A digraph is called \textit{kernel-perfect} if each of its induced subgraphs has a kernel.
\end{defi}

\begin{thm} (Bondy-Boppana-Siegel, see \cite{MR1179249},\cite{MR1309363}) \label{thm:BBS}
    Let $G=(V,E)$ be a graph, let $D$ be a kernel perfect biorientation of $G$, and define $f:V(G)\rightarrow \mathbb{N}$ by $f(v)= d^+_D(v)+1$ for each $v\in V$. Then $G$ is $f$-choosable.
\end{thm}

\subsection{DP-Colorings and Derangement Assignments}

The following definition was first given in \cite{MR3758240} and was first considered for multigraphs in \cite{MR3686937}.

\begin{defi} \label{def:cor}
  Let $G=(V,E)$ be a multigraph. \vspace{.2cm} 

\begin{itemize}

\item A \textit{correspondence assignment} for $G$ consists of an assignment of lists $L=\{L(v)\}_{v\in V}$ and set of partial matchings $C=\{C_e\}_{e\in E}$ that assigns to every edge $e=\{u,v\}$ a partial matching $C_e$ between $\{u\}\times L(u)$ and $\{v\}\times L(v)$.\vspace{.2cm} 

\item An $(L,C)$-coloring of $G$ is a function $\varphi$ that to each vertex $v\in V(G)$ assigns a color $\varphi(v)\in L(v)$, such that for every $e=\{u,v\}\in E(G)$, the vertices $(u,\varphi(u))$ and $(v,\varphi(v))$ are non-adjacent in $C_e$. If such an $(L,C)$ coloring exists, we say that $G$ is $(L,C)$-colorable. 

\end{itemize}

 Colorings of the above type are called \textit{DP-colorings}. 
\end{defi}

\Cref{fig:badcor} shows a correspondence assignment $(L,C)$ on $C_4$, with $L(v)=\{1,2\}$ for each $v\in V(C_4)$. As has become the custom, we represent each list $L(v)$ ``inside" its corresponding vertex, and draw the matchings for each edge.
\begin{figure}[h]

\caption{}
\vspace{.1cm}
\label{fig:badcor} \begin{center}
\begin{tikzpicture}[scale=1.8]
\node at (1.75,.5){} ;

\node at (-.2,-.2){1};
\node at (-.4,-.4){2};
\node at (-.2,1.2){1};
\node at (-.4,1.4){2};
\node at (1.2,1.2){1};
\node at (1.4,1.4){2};
\node at (1.2,-.2){1};
\node at (1.4,-.4){2};

\draw[] (-.2,-.1)--(-.2,1.1);
\draw[] (-.4,-.3)--(-.4,1.3);
\draw[] (-.1,1.2)--(1.1,1.2);
\draw[] (-.3,1.4)--(1.3,1.4);
\draw[] (-.1,-.2)--(1.1,-.2);
\draw[] (-.3,-.4)--(1.3,-.4);

\draw[] (1.2,1.1)--(1.4,-.3);
\draw[] (1.4,1.3)--(1.2,-.1);


the\draw[rotate around={45:(-.3,1.3)},gray] (-.3,1.3) ellipse (5pt and 10pt);
\draw[rotate around={-45:(1.3,1.3)},gray] (1.3,1.3) ellipse (5pt and 10pt);
\draw[rotate around={-45:(-.3,-.3)},gray] (-.3,-.3) ellipse (5pt and 10pt);
\draw[rotate around={45:(1.3,-.3)},gray] (1.3,-.3) ellipse (5pt and 10pt);
\end{tikzpicture}
\end{center}
\end{figure}

The correspondence assignment in \Cref{fig:badcor} demonstrates that the Bondy-Boppana-Siegel lemma does not apply to DP-coloring if we replace list colorings with DP-colorings in the statement of the lemma. 
Since the cyclic orientation $D$ of $C_4$ is kernel-perfect, and has $d^+_{D}(v)=1$ for each $v$, and the above correspondence assignment is such that each list $L(v)$ has length $d^+_D(v)+1=2$, but $G$ is not $(L,C)$-colorable. 
 
List colorings are the special case of DP-colorings where each partial matching $C_{\{u,v\}}$ is given by $\{(u,c_1),(v,c_2)\}\in C_{\{u,v\}}$ if and only if $c_1=c_2$. Following the terminology in \cite{MR3758240}, if $(L,C)$ is a correspondence assignment on $G$ and $e=\{u,v\}\in E(G)$ is such that $\{(u,c_1),(v,c_2)\}\in C_{e}$ implies $c_1=c_2$, then $e$ is said to be a \textit{straight} edge (with respect to $(L,C)$). If $e\in E(G)$ is not straight, we will say that $e$ is \textit{twisted} (with respect to $(L,C))$.  

Given a correspondence assignment $(L,C)$ on a graph $G$, let $E_S(G)\subseteq E(G)$ denote the set of straight edges of $G$ and let $E_T(G)\subseteq E(G)$ denote the set of twisted edges. Define $G_S=(V(G),E_S(G))$, and define $G_T=(V(G),E_T(G))$.  

\begin{defi}\label{def:derangement}
    Let $(L,C)$ be a correspondence assignment on a multigraph $G$ and let $e=\{v,w\}$ be an edge of $G$. The matching $C_{e}$ is called a \textit{partial derangement} if for all $c\in L(v)\cap L(w)$, $\{(v,c),(w,c)\}\notin C_e$. If, for every twisted edge $e\in E_T(G)$, we have that $C_e$ is a partial derangement, we call $(L,C)$ a \textit{derangement assignment} on $G$.
\end{defi}

\section{An Orientation Theorem for Derangement Assignments}

The following theorem is a generalization of \Cref{thm:BBS} to DP-colorings from derangement assignments. 

\begin{thm}\label{thm:BBSDP} Let $G$ be a multigraph, and $(L,C)$ a derangement assignment on $G$. If there exists a biorientation $D$ of $G$ such that $D$ induces a kernel-perfect biorientation of $G_S$, $E_T(G)\subseteq E_{D2}(G)$, and $|L(v)|\geq d^+_D(v)+1$ for each $v\in V(G)$, then $G$ is $(L,C)$-colorable. \end{thm}

\begin{proof} We prove the claim by induction on $|V(G)|$. The claim is trivial if $|V(G)|=1$. Suppose that $V(G)=n\geq 2$ and that the claim holds for all graphs on fewer than $n$ vertices. Let $(L,C)$ be a derangement assignment for $G$, and suppose $D$ is a biorientation such that $E_T(G)\subseteq E_{D2}(G)$, $D$ induces a kernel perfect biorientation on $G_S$, and that $|L(v)|\geq d^+_D(v)+1$ for each $v\in V(G)$. Let $a\in \bigcup_{v\in V}L(v)$, and let $V_a=\{v\in V:a\in L(v)\}$. Then $G_S[V_a]$ has a kernel, $U$. Note that if we consider $G[U]$ with the correspondence assignment induced on it by $(L,C)$, then we can assign the color $a$ to each $u\in U$ to get a correspondence coloring of $G[U]$, because the only edges in $G[U]$ are twisted, and are therefore partial derangements. 

Now, for each $w\in V(G)\setminus U$, define  $$B(w)= \{b\in L(w): \exists e=\{w,u\}\in E(G) \text{ with } u\in U \text{ and }\{(w,b),(u,a)\}\in C_{e}\},$$ and define $L'(w)=L(w)\setminus B(w)$. 
 
 Fix $w\in V\setminus U$. We observe that there are at least $|B(w)|$ out-edges (in $D$) from $w$ to vertices in $U$: Suppose $b\in B(w)$. If $\{(w,b),(u,a)\}\in C_{e}$ and $e=\{w,u\}\in E_S(G)$ with $u\in U$, then $b=a$, so $w\in V_a$, and therefore there exists an outgoing edge $\vec{e}_b=(w,u_b)$ from $w$ to some vertex $u_b\in U$, because $U$ is a kernel of $G[V_a]$. 
 Now, if $b'\in B(w)$, and $\{(w,b'),(u_{b'},a)\}\in C_{e}$ for some $e=\{w,u_{b'}\}$ with $u_{b'}\in U$, and $e=\{w,u_{b'}\}\in E_T(G)$, then
 $\vec{e}_{b'}=(w,u_{b'})$ is an edge of $D$ because $E_T(G)\subseteq E_{D2}(G)$. Furthermore, if $b_1,b_2\in B(w)$ and $b_1\neq b_2$, then $\vec{e}_{b_1}\neq \vec{e}_{b_2}$ because each $C_e$ is a matching. Hence there are at least $|\{\vec{e}_b:b\in B(w)\}|=|B(w)|$ distinct out-edges from $w$ to vertices in $U$. Therefore, $d^+_{D-U}(w)+1\leq |L'(w)|$ for each $w\in V(G-U)$.

Now, consider $G-U$ with the correspondence matchings induced by $C$ on the restricted vertex lists $L'(w)$, and call this induced correspondence assignment $C'$.  Note that restricting the edge matchings to the new lists $L'$ can possibly make an edge that was twisted with respect to $(L,C)$ into an edge that is straight with respect to $(L',C')$ if its matching becomes empty, but since empty matchings do not create any restrictions on the coloring, we can ignore these edges; let $E^*$ be the set of edges in $G-U$ that are twisted with respect to $(L,C)$ and straight with respect to $(L',C')$. Let $(G-U)^*$ be the graph with vertex set $V\setminus U$ and edge set $E(G-U)\setminus E^*$. Let $C^*\subseteq C'$ be the correspondence assignment given by restricting $C'$ only to the edges of $(G-U)^*$, and consider the correspondence assignment $(L',C^*)$ on $(G-U)^*$.  Then $(G-U)^*_S$ is an induced subgraph of $G_S$, and therefore the biorientation on $(G-U)^*_S$ induced by $D$ is kernel-perfect. Hence, by the inductive hypothesis, there is an $(L',C^*)$-coloring of $(G-U)^*$, and the same coloring is an $(L',C')$-coloring of $G-U$.

This coloring of $G-U$ does not conflict with our coloring of the vertices of $U$ (with respect to the original correspondence assignment $C$), because we deleted all of the colors that were matched to the color $a$ when we formed the lists $L'$. Hence, combining our colorings of $U$ and $G-U$, we get an $(L,C)$ coloring of $G$, as desired. 
\end{proof}

If $G$ is equipped with a correspondence assignment $(L,C)$ with the standard list correspondence induced by the names of the list elements, then $E_T=\emptyset$, and the statement of \Cref{thm:BBSDP} just becomes the statement of the original Bondy-Boppana-Siegel lemma for list colorings.

Furthermore, M. Richardson showed in \cite{MR0075184} that any orientation without directed odd cycles has a kernel.  Since every induced subgraph of such an orientation will also be without directed odd cycles, an orientation of this type is also kernel-perfect. Hence, if $G$ is a graph equipped with a derangement assignment and $D$ is an orientation of $G$ such that every directed odd cycle contains a twisted edge, we have that $G_S$ is kernel perfect. This yields the following corollary of \Cref{thm:BBSDP}. 

\begin{cor}\label{cor:richextended}
    Suppose $G$ is a graph and $(L,C)$ is a derangement assignment on $G$. If there exists and orientation $D$ of $G$ such that $E_T\subseteq E_{D2}$, $|L(v)|\geq d^+_{D}(v)+1$ for each $v\in V$, and every odd directed cycle of $D$ contains a twisted edge, then $G$ is $(L,C)$-colorable. 
\end{cor}

It is worth noting that the proof in \cite{MR0075184} gives a polynomial time algorithm for finding kernels in directed graphs with no odd directed cycles, and that this algorithm, combined with the coloring process given by the proof of \Cref{thm:BBSDP}, yields an algorithm to DP-color graphs satisfying the hypothesis of \Cref{cor:richextended}. 

\section{An Application to Colorings of Signed Graphs}
In this section, we give an application of \Cref{thm:BBSDP} to signed graphs. A signed graph is a pair $(G,\sigma)$, where $G$ is a graph and $\sigma:E(G)\rightarrow \{-1,1\}$.  A (signed) coloring in $k$-colors of $(G,\sigma)$, as defined in \cite{MR0675866}, is a function $\psi:V(G)\rightarrow \{-k,-k+1,...-1,0,1,...,k-1,k\}$ such that for each $e=\{v,w\}\in E(G)$, we have that $\psi(v)\neq \sigma(e)\psi(w)$. A signed coloring is called zero-free if for all $v\in V$, $\psi(v)\neq 0$.

In \cite{JIN2016234}, the idea of signed colorings was extended to signed list colorings. Given a signed graph, $(G,\sigma)$, and an assignment of lists $L=\{L(v)\}_{v\in V}$, with $L(v)\subseteq \mathbb{Z}$ for each $v$, an $(L,\sigma)$-coloring is a function $\psi:V(G)\rightarrow \bigcup_{v\in V}L(v)$ such that $\psi(v)\in L(v)$ for each $v$, and for each $e=\{v,w\}\in E(G)$, we have that $\psi(v)\neq \sigma(e)\psi(w)$. If $(G,\sigma)$ admits an $(L,\sigma)$-coloring, we say $(G,\sigma)$ is $L$-colorable. 

Given a signed graph $(G,\sigma)$, let $E^+(G)\subset E(G)$ be the set of edges $e\in E(G)$ such that $\sigma(e)=1$, and $E^-(G)\subseteq E(G)$ be the set of edges $e\in E(G)$ such that $\sigma(e)=-1$. Define $G^+=(V,E^+(G))$, and $G^-=(V,E^-(G))$.

\begin{cor} \label{cor:BBSsigned} Suppose that $(G,\sigma)$ is a signed graph and $D$ is a biorientation of $G$ such that each $e\in E^-(G)$ is oriented in both directions in $D$, and $D$ induces a kernel perfect biorientation of $G^+$. If $L=\{L(v)\}_{v\in V(G)}$ is a list assignment such that $L(v)\subseteq \mathbb{Z}\setminus \{0\}$ and $|L(v)|\geq d^+_{D}(v)+1$ for each $v\in V(G)$, then $(G,\sigma)$ is $L$-colorable.  \end{cor}
\begin{proof} Given a signed graph $(G,\sigma)$, and a list assignment $L=\{L(v)\}_{v\in V}$, with $L(v)\subseteq \mathbb{Z}\setminus \{0\}$ for each $v$, define the partial matchings $C_e$ as follows. If $e=\{v,w\}$ and $\sigma(e)=1$, then $\{(v,c_1),(v,c_2)\}\in C_e$ if and only if $c_1=c_2$. If $\sigma(e)=-1$, then $\{(v,c_1),(w,c_2)\}\in C_e$ if and only if $c_1=-c_2$. Then a function $\varphi:V\rightarrow \bigcup_{v\in V}L(v)$ is an $(L,C)$-coloring of $G$ if and only if it is a signed $L$-coloring. Furthermore, if $e\in E(G)$ is such that $\sigma(e)=1$,then $e$ is straight with respect to $(L,C)$, and if $\sigma(e)=-1$, then $C_e$ is a partial derangement, since $n\neq -n$ for all $n\in \mathbb{Z}\setminus \{0\}$. Since every twisted edge in $G$ is a partial derangement, $(L,C)$ is a derangement assignment, and since $E_T(G)\subseteq E^-(G)$, the claim follows from \Cref{thm:BBSDP}.

\end{proof}

Analogously to \Cref{cor:richextended}, we have the following for zero-free signed list colorings. 

\begin{cor} \label{cor:richsigned} Suppose that $(G,\sigma)$ is a signed graph and $D$ is a biorientation of $G$ such that each $e\in E^-(G)$ is oriented in both directions in $D$, and that $D$ induces an orientation on $G^+$ that has no odd directed cycles. If $L=\{L(v)\}_{v\in V(G)}$ is a list assignment such that $L(v)\subseteq \mathbb{Z}\setminus \{0\}$ and $|L(v)|\geq d^+_{D}(v)+1$ for each $v\in V(G)$, then $(G,\sigma)$ is $L$-colorable.  \end{cor}
\section{Further Discussion on Derangement Assignments} \label{sec:mot}

We discuss some questions that may motivate further study of derangement assignments and their differences from standard list colorings. Let $G=(V,E)$ be a graph, and suppose that $E_1,E_2\subseteq E$ with $E=E_1 \cup E_2$ and $E_1\cap E_2 =\emptyset$. Let $G_1=(V,E_1)$, and $G_2=(V,E_2)$. Suppose that we know that $G_1$ is $f_1$-choosable, for some function $f_1:V\rightarrow \mathbb{Z}$. It is interesting to try to understand how we might use the $f_1$-choosability of $G_1$ to make some conclusion about the choosability of $G$, as a whole. For example, one might ask the following question. 
If $G_1$ is $f_1$-choosable, is it true that $G$ is $f$-choosable, where $f:V\rightarrow \mathbb{Z}$ is defined by $f(v)=f_1(v)+d_{G_2}(v)$? The answer to this question, in general, is \textit{no}, as can be seen from the factorization of $K_4$ shown in \Cref{fig:k4edgedecomp}; $G_1$ is $2$-choosable, and $d_{G_2}(v)=1$ for each $v\in V$, but $G=K_4$ is not $3$-choosable.  However, if we instead consider derangement assignments on $G$ such that $E_1=E_S(G)$ and $E_2=E_T(G)$, we will see that, as a consequence of \Cref{thm:BBSDP}, we get a more pleasing answer.


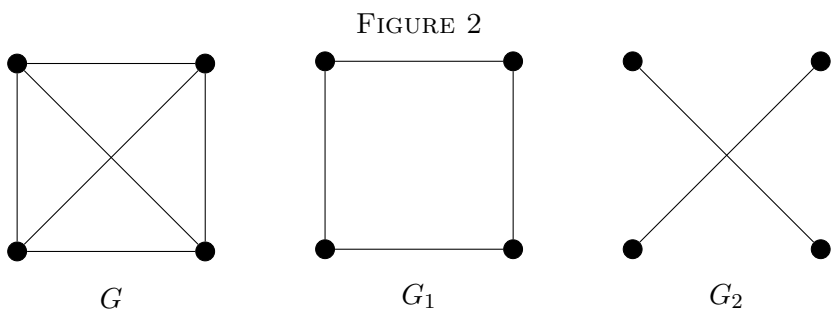
\vspace{.1cm}
\begin{figure}[h]
\caption{}
\label{fig:k4edgedecomp}
\begin{center}
\begin{tikzpicture}[scale=2.5]

\tikzset{vertex/.style = {shape=circle,fill=black, draw, inner sep=0pt,minimum size=2.5mm}}
\tikzset{edge/.style = {-
}}

\node[] at (.5,-.25) {$G$};
\node[vertex] (a) at  (0,0) {};
\node[vertex] (b) at  (0,1) {};
\node[vertex] (c) at  (1,1) {};
\node[vertex] (d) at  (1,0) {};
\draw[edge] (a) to (b);
\draw[edge] (b) to (c);
\draw[edge] (c) to (d);
\draw[edge] (d) to (a);

\draw[edge] (a) to (c);
\draw[edge] (b) to (d);

\end{tikzpicture}\hspace{1.2cm}
\begin{tikzpicture}[scale=2.5]

\tikzset{vertex/.style = {shape=circle,fill=black, draw, inner sep=0pt,minimum size=2.5mm}}
\tikzset{edge/.style = {-
}}
\node[] at (.5,-.25) {$G_1$};
\node[vertex] (a) at  (0,0) {};
\node[vertex] (b) at  (0,1) {};
\node[vertex] (c) at  (1,1) {};
\node[vertex] (d) at  (1,0) {};
\draw[edge] (a) to (b);
\draw[edge] (b) to (c);
\draw[edge] (c) to (d);
\draw[edge] (d) to (a);


\end{tikzpicture}\hspace{1.2cm}
\begin{tikzpicture}[scale=2.5]

\tikzset{vertex/.style = {shape=circle,fill=black, draw, inner sep=0pt,minimum size=2.5mm}}
\tikzset{edge/.style = {-
}}
\node[] at (.5,-.25) {$G_2$};
\node[vertex] (a) at  (0,0) {};
\node[vertex] (b) at  (0,1) {};
\node[vertex] (c) at  (1,1) {};
\node[vertex] (d) at  (1,0) {};

\draw[edge] (a) to (c);
\draw[edge] (b) to (d);

\end{tikzpicture}
\end{center} 
\end{figure}


We relate the above question to biorientations. If we equip $G_1$ and $G_2$ with the biorientations $D_1$ and $D_2$ in \Cref{fig:k4edgedirected}, then $d^+_{D_1}(v)=1$ and $d^+_{D_2}(v)=d_{G_2}(v)$ for each $v\in V$. Thus, if $f_1$ is defined by $f_1(v)=d^+_{D_1}(v)+1=2$, we see that $f(v)=f_1(v)+d_{G_2}(v)=d^+_D(v)+1$, for each $v$, where $D$ is the biorientation of $G$ such that $\vec{E}(D)=\vec{E}(D_1)\cup \vec{E}(D_2)$.  Hence, with this setup, our question becomes one about DP-colorability from lists of length at least $d^+_{D}(v)+1$, where $D$ is a biorientation with the set of edges $E_{D2}$ specified. 

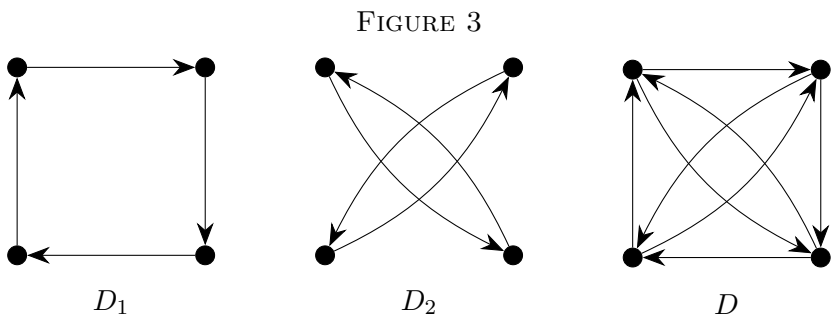
\begin{figure}[h]
\caption{}
\label{fig:k4edgedirected}
\vspace{.1cm}
\begin{center}

\begin{tikzpicture}[scale=2.5]

\tikzset{vertex/.style = {shape=circle,fill=black, draw, inner sep=0pt,minimum size=2.5mm}}
\tikzset{edge/.style = {-{Stealth[scale=1.8]}
}}
\node[] at (.5,-.25){$D_1$};
\node[vertex] (a) at  (0,0) {};
\node[vertex] (b) at  (0,1) {};
\node[vertex] (c) at  (1,1) {};
\node[vertex] (d) at  (1,0) {};
\draw[edge] (a) to (b);
\draw[edge] (b) to (c);
\draw[edge] (c) to (d);
\draw[edge] (d) to (a);


\end{tikzpicture}\hspace{1.2cm}
\begin{tikzpicture}[scale=2.5]

\tikzset{vertex/.style = {shape=circle,fill=black, draw, inner sep=0pt,minimum size=2.5mm}}
\tikzset{edge/.style = {-{Stealth[scale=1.8]}
}}
\node[] at (.5,-.25){$D_2$};
\node[vertex] (a) at  (0,0) {};
\node[vertex] (b) at  (0,1) {};
\node[vertex] (c) at  (1,1) {};
\node[vertex] (d) at  (1,0) {};

\draw[edge] (a) to[bend right= 20] (c);
\draw[edge] (c) to[bend right= 20] (a);
\draw[edge] (b) to[bend right= 20] (d);
\draw[edge] (d) to[bend right= 20] (b);

\end{tikzpicture}\hspace{1.2cm}
\begin{tikzpicture}[scale=2.5]

\tikzset{vertex/.style = {shape=circle,fill=black, draw, inner sep=0pt,minimum size=2.5mm}}
\tikzset{edge/.style = {-{Stealth[scale=1.8]}
}}
\node[] at (.5,-.25){$D$};
\node[vertex] (a) at  (0,0) {};
\node[vertex] (b) at  (0,1) {};
\node[vertex] (c) at  (1,1) {};
\node[vertex] (d) at  (1,0) {};
\draw[edge] (a) to (b);
\draw[edge] (b) to (c);
\draw[edge] (c) to (d);
\draw[edge] (d) to (a);

\draw[edge] (a) to[bend right= 20] (c);
\draw[edge] (c) to[bend right= 20] (a);
\draw[edge] (b) to[bend right= 20] (d);
\draw[edge] (d) to[bend right= 20] (b);

\end{tikzpicture}
\end{center} 
\end{figure}

Since $D_1$ in \Cref{fig:k4edgedirected} is a kernel-perfect orientation such that $f_1(v)\geq d^+_{D_1}(v)+1$ for each $v\in V$, we see that even if we restrict to functions $f_1$ such that $G_1$ is $f_1$-choosable and such that there exists some kernel-perfect orientation $D'$ of $G_1$ with $f_1(v)\geq d^+_{D'}(v)+1$ for all $v\in V$, the answer to our question is still \textit{no}. However, in the case of derangement assignments, we have the following corollary of \Cref{thm:BBSDP}.

\begin{cor} \label{cor:chooseversion}Let $(L,C)$ be a derangment assignment on a multigraph $G=(V,E)$. Suppose that $f_S:V\rightarrow \mathbb{N}$ is such that there exists a kernel-perfect biorientation $D_S$ of $G_S$ with $f_S(v)\geq d^+_{D_S}(v)+1$ for each $v\in V$. If $|L(v)|\geq f_S(v)+d_{G_T}(v)$ for each $v\in V$, then $G$ is $(L,C)$-colorable.
\end{cor}

We wonder if \Cref{thm:BBSDP} and \Cref{cor:chooseversion} can be strengthened and if other closely related statements might be true. We pose two questions.

  \begin{que} Suppose $(L,C)$ is a derangement assignment on $G$,  $f_S:V\rightarrow \mathbb{N}$, and that $G_S$ is $f_S$ choosable. Is it true that if $|L(v)|\geq f_S(v)+d_{G_T}(v)$ for each $v$, then $G$ must be $(L,C)$-colorable?  
  \end{que}

  \begin{que} Suppose that $D$ is a kernel-perfect biorientation of $G$, and $(L,C)$ is a derangement assignment on $G$ such that $E_T(G)\subseteq E_{D2}$. Is it true that if $|L(v)|\geq d^+_{D}(v)+1$ for each $v\in V$, then $G$ must be $(L,C)$-colorable? 
  \end{que}
  
Recall also that that DP-colorings generalize several earlier variants of coloring problems. It would be interesting to find out if there are other previously defined variants of colorings, other than zero-free signed colorings, that can be seen as special cases of derangement assignments. Finally, we note that in \cite{MR3758240}, a notion of \textit{equivalent} correspondence assignments was introduced: two correspondence assignments $(L,C)$ and $(L',C')$ are equivalent if one can be obtained from the other by `renaming' of the list elements, while preserving the structure of the matchings. If $(L,C)$ and $(L',C')$ are equivalent correspondence assignments on a graph $G$, then $G$ is $(L,C)$-colorable if and only if it is $(L',C')$-colorable. Thus, it would be interesting to study whether there is some way of methodically relabeling arbitrary correspondence assignments for specified classes of graphs so that the results from this work might be applicable.

\bibliographystyle{plain}
\bibliography{references}     

\vspace{1cm}
 
Email: igossett@oberlin.edu\

Department of Mathematics, Oberlin College, Oberlin, OH, 44074, USA.

\end{document}